\documentclass[12pt]{amsart}
\usepackage[left=3cm,right=3cm,top=4cm,bottom=4cm,a4paper]{geometry}
\usepackage{amsmath,amssymb,amsfonts,amsthm,latexsym}
\usepackage{enumerate}
\usepackage{thmtools}
\usepackage{graphicx}
\usepackage[dvips]{epsfig}
\usepackage[colorlinks=true, urlcolor=blue, linkcolor=blue, citecolor=blue]{hyperref}
\numberwithin{equation}{section}
\newtheorem{thm}{Theorem}
\numberwithin{thm}{section}
\newtheorem{defn}[thm]{Definition}

\newtheorem{prop}[thm]{Proposition}
\newtheorem{lemma}[thm]{Lemma}

\newtheorem{rmk}[thm]{Remark}
\newtheorem{ex}[thm]{Example}

\begin{document}

\title{On the scaling methods by Pinchuk and Frankel}

\author{Seungro Joo}
\address{(Joo) Department of Mathematics, POSTECH, 
Pohang 790-784 The Republic of Korea}%
\email{beartan@postech.ac.kr}
\thanks{This research was supported in part by the Grant 2011-0030044 (The SRC-GAIA, PI : K.-T. Kim) of the National Research Foundation of The Republic of Korea.}

\begin{abstract}
The main purpose of this paper is to study two scaling methods developed respectively by Pinchuk and Frankel. We introduce first a continuously-varying global coordinate system, and give an alternative proof to the convergence of Pinchuk's scaling sequence (and of our modification) on bounded domains with finite type boundaries in $\mathbb{C}^2$. Using this, we discuss the modification of the Frankel scaling sequence on nonconvex domains. We also observe that two modified scalings are equivalent.
\end{abstract}

\maketitle

\section{Introduction}

The scaling methods were introduced by Pinchuk \cite{p} and Frankel \cite{f} independently in 1980's as a technique to study bounded domains in $\mathbb{C}^n$ with noncompact automorphism group. These techniques have been developed further by many authors and have become an important tool to prove the results of \cite{p}, \cite{bp1}, \cite{k} and others.

Pinchuk's scaling sequence was constructed by a sequence of compositions of stretching maps, say $\Lambda_j$, and automorphisms $\phi_j$ of the given domain $\Omega$. If the automorphism group of $\Omega$ is noncompact then, in many cases, there is a sequence $\{\phi_j\} \subset Aut(\Omega)$ which contracts compact subsets successively to some boundary point. On the other hand, a sequence of stretching maps is a divergent sequence of shear maps, the composition of $\mathbb{C}$-affine maps and triangular maps. The general expectation is that there is a subsequence of the Pinchuk scaling sequence $\{ \sigma_j := \Lambda_j \circ \phi_j \}$ convergent to the limit map, say $\widehat{\sigma}$, uniformly on compact subsets of $\Omega$. If this limit were 1-1, then it would be a re-embedding of $\Omega$ into $\mathbb{C}^n$. If $\Omega$ is a domain in $\mathbb{C}$ with smooth boundary, then the image of the limit map turns out to be a half plane. This is the special case of the Riemann mapping theorem and therefore it seems natural to hope for the convergence of $\{ \sigma_j \}$ in all dimensions.

As the first result in the higher dimensions, Pinchuk proved that his scaling sequence has a subsequence that converges to a biholomorphism uniformly on compact subsets for the class of bounded strongly pseudoconvex domains in all dimensions \cite{p}; this proves Wong-Rosay theorem \cite{w,r}. And later, Bedford and Pinchuk \cite{bp1} showed the convergence of the sequence if the domain is bounded with a finite type boundary in the sense of D'Angelo \cite{d}.

One of the difficulties in proving the convergence is that the expected limit domain $\widehat{\Omega}$ is not bounded; its Kobayashi hyperbolicity is not {\it a priori} clear. Pinchuk considers, alternatively, the convergence of the backward scaling sequence $\{ \sigma_j^{-1} \}$. If the limit domain $\widehat{\Omega}$ is well-defined in some sense, then this sequence $\{ \sigma_j^{-1} \}$ always has a convergent subsequence by Montel's theorem. Now, a question arises naturally: is the limit map of the backward sequence 1-1? For the class of bounded domains in $\mathbb{C}^2$ whose boundaries are of finite type, Bedford and Pinchuk have given a general affirmative resolution \cite{bp1} (cf. \cite{bc} also). If the limit map is surjective, additionally, then it follows that the inverse of the limit map is actually the same as a subsequential limit of the initial Pinchuk scaling sequence. This establishes the Pinchuk scaling method.

The Frankel scaling sequence follows the same principle but its construction is different. Given a domain $\Omega$, a point $p \in \Omega$ and a sequence of automorphisms $\{ \phi_j \}$, it is defined directly by $\omega_j(\mathbf{z}) := [d\phi_j|_p]^{-1} (\phi_j (\mathbf{z}) - \phi_j (p))$. If $\{\phi_j (p)\}$ converges to some boundary point of $\Omega$, then $\{\phi_j\}$ cannot converge to another automorphism. In fact, $\lim_{j \rightarrow \infty} \det{(d\phi_j|_p)} = 0$. Then $[d\phi_j|_p]^{-1}$ diverges. So, Frankel's scaling method appears to be similar to Pinchuk's. The sequence $\{[d\phi_j|_p]^{-1}\}$ stretches in some sense, whereas the sequence $\{\phi_j\}$ contracts. Now one can naturally pose the question: when does Frankel's scaling sequence form a normal family? Frankel proved that it suffices for $\Omega$ to be convex and Kobayashi hyperbolic \cite{f}.
\medskip

The purpose of this article is summarized into the following:
\medskip

In Section 2, we introduce a special continuously-varying coordinate system, pertaining to the target boundary point. Using this coordinate system, we give another proof to the convergence of the Pinchuk scaling sequence on a bounded domain in $\mathbb{C}^2$ with smooth finite-type boundary. We feel that our proof is simpler and more straightforward than that of Berteloot/C\oe ur\'{e} \cite{bc}.

Section 3 concerns the Frankel scaling sequence. The convexity was essential for its convergence to a holomorphic embedding into $\mathbb{C}^n$. There has been a question whether it converges without convexity. Here, we give a modification of the Frankel scaling sequence so that they may converge also on some nonconvex domains, using a sequence $\{\psi_j\}$ of automorphisms of $\mathbb{C}^n$ that converges to another. Two examples are given to show several aspects of the (modified) Frankel scaling sequence.
\smallskip

Finally in Section 4, we observe that the limit maps, if they exist, of Pinchuk and modified Frankel's scaling sequences are equivalent. Notice that this generalizes a theorem of Kim/Krantz in \cite{kk} for the convex case.

In Appendix, we give a proof of an existence of the coordinate system introduced in Section 2.
\medskip

\section{The Pinchuk scaling sequence}

Recall the definition of the finite type in the sense of D'Angelo \cite{d}.
\begin{defn} \rm
Let $\Omega$ be a domain in $\mathbb{C}^n$ with smooth boundary. Let $q$ be a point in $\partial \Omega$ and $\rho$ be a local defining function of $\Omega$ at $q$. The {\it type} $\Delta (q) = \Delta (\Omega, q)$ at $q$ is the positive value defined by
\begin{equation} \nonumber
\Delta (q) := \sup_{h} \frac{\nu(\rho \circ h)}{\nu(h - q)},
\end{equation}
where $\nu(f)$ is the order of vanishing of $f$ at $0$ and the supremum is taken over all nontrivial analytic discs $h$ in $\mathbb{C}^n$ with $h(0) = q$.

The point $q$ is called a {\it finite type boundary point} of $\Omega$ if $\Delta (q)$ is finite. If all the boundary points of $\Omega$ is of finite type, then $\Omega$ is called a {\it domain with finite type boundary}.
\end{defn}

From now on, the main subject is the bounded domain $\Omega$ with noncompact automorphism group. Under this condition, there is a point $p \in \Omega$ and a sequence $\{\phi_j\} \subset Aut(\Omega)$ with $\lim_{j \to \infty} \phi_j(p) = \widehat{p}$ for some boundary point $\widehat{p}$. We call such sequence $\{\phi_j(p)\}$ a {\it boundary accumulation automorphism orbit} converging to $\widehat{p}$, and present the following improvement upon the scaling theorem of Pinchuk.
\begin{thm}
Let $\Omega$ be a bounded domain in $\mathbb{C}^2$ with smooth pseudoconvex boundary. Assume that $Aut(\Omega)$ admits a boundary accumulating automorphism orbit $\{\phi_j(p)\}$ converging to $\widehat{p} \in \partial \Omega$. If $\widehat{p}$ is of finite type in the sense of D'Angelo, then there is a sequence $\{ \Lambda_j \}$ in $Aut(\mathbb{C}^2)$ such that the sequence $\{ \Lambda_j \circ \phi_j \}$ has a subsequence that converges to a biholomorphism-into $\mathbb{C}^2$ uniformly on compact subsets of $\Omega$. Moreover, the image of $\Omega$ by the limit map is of the form $\{ (w, z) \in \mathbb{C}^2 \mid \textrm{Re}\,w + P(z, \bar{z}) < 0\}$ for some subharmonic polynomial $P$ with no harmonic terms.
\end{thm}

This stretching sequence $\Lambda_j$ is an automorphism of $\mathbb{C}^2$. Indeed, the map $\Lambda_j$ is a composition of shear maps (cf. Section 2.2). This property plays an important role in proving Theorem 4.1. We call this $\Lambda_j \circ \phi_j$ the ($j$-th) {\it Pinchuk scaling map} of the {\it Pinchuk scaling sequence} $\{ \Lambda_j \circ \phi_j \}$.
\bigskip

\subsection{An admissible coordinate system for finite type boundary.} To prove Theorem 2.2, we introduce a continuously-varying coordinate system near the target boundary point.
\begin{restatable}{prop}{keylemma}
Let $\Omega$ be a domain in $\mathbb{C}^2$ with a smooth boundary. Fix a boundary point $p$ and an integer $r > 0$. Assume that the outward unit normal vector of $\partial\Omega$ at $p$ is $(1, 0)$. Then there is a neighborhood $U$ of $p$ and a continuous map $\Psi \colon (\partial\Omega \cap U) \times \mathbb{C}^2 \rightarrow \mathbb{C}^2$ which satisfies, for each $q \in \partial\Omega \cap U$, the following properties:
\begin{enumerate}
\item The map $\Psi_q := \Psi (q, \cdot)$ is the composition of a translation, a dilation and a triangular map.
\item $\Psi_q (q) = (0, 0)$.
\item The local defining function of $\Psi_q (\Omega \cap U)$ at $(0, 0)$ is represented by
\begin{equation} \nonumber
\left\{ (w, z) \mid \textrm{Re}\,w + P(z, \bar{z}) + R(z, \bar{z}) + \textrm{Im}\left( \frac{w}{c} \right) Q \left( \textrm{Im}\left( \frac{w}{c} \right), z, \bar{z} \right) < 0 \right\}
\end{equation}
where $P$ is a real-valued polynomial of degree $r$ with no harmonic terms, $c = c(q)$ is a constant satisfying $\textrm{Re}\,c \neq 0$ and $\lim_{q \to p} c(q) = 1$, $R$ and $Q$ are real-valued smooth functions with the conditions on the vanishing order $\nu(R(z, \bar{z})) > r$ and $\nu \left(Q \left( \textrm{Im}\left( \frac{w}{c} \right), z, \bar{z} \right) \right) \ge 1$.
\end{enumerate}
Moreover, if the point $q \in \partial\Omega$ is pseudoconvex of finite type $2k$ and $r = 2k$, then $P$ is a homogeneous polynomial of degree $2k$.
\end{restatable}
This coordinate system is a variation of that of \cite{c}. We shall present the proof of this in Appendix.
\bigskip

\subsection{Construction of the scaling sequence} Let $\Omega$ be a domain in the hypothesis of Theorem 2.2. Denote by $\widehat{p} := \lim_{j\rightarrow \infty} \phi_j(p)$, the orbit accumulation boundary point. Taking a unitary transformation if necessary, one may assume that the outward unit normal vector of $\partial\Omega$ at $\widehat{p}$ is $(1, 0)$. Set $r = 2k$, which is the same as the type at $\widehat{p}$, and apply Proposition 2.3 to $\widehat{p}$. Then there is a global coordinate map $\Psi_{\widehat{p}}^{\Omega} \in Aut(\mathbb{C}^2)$ with $\Psi_{\widehat{p}}^{\Omega}(\widehat{p}) = (0, 0)$ and a neighborhood $U$ of $\widehat{p}$ such that the local defining function of $\Psi_{\widehat{p}} (\Omega)$ at $(0, 0)$ is represented by:
$$
\Psi_{\widehat{p}}^{\Omega}(\Omega \cap U) = \{ (w, z) \in \Psi_{\widehat{p}}^{\Omega}(U) \mid \rho_{\widehat{p}} (w, z) < 0 \}
$$
where:
\begin{enumerate}
\item
$\rho_{\widehat{p}} (w, z) = \textrm{Re}\,w + P(z, \bar{z}) + R(z, \bar{z}) + \textrm{Im}\,w \cdot Q(\textrm{Im}\,w, z, \bar{z})$,
\item
$P$ is a real-valued homogeneous polynomial of degree $2k$ without harmonic terms,
\item $R$ and $Q$ are real-valued smooth functions with the conditions $\nu(R(z, \bar{z})) > 2k$ and $\nu(Q(\textrm{Im}\,w, z, \bar{z})) \ge 1$.
\end{enumerate}

Taking a subsequence if necessary, one may assume that $\{ \phi_j(p) \} \subset U$. Write $p_j := \Psi_{\widehat{p}}^{\Omega}(\phi_j(p))$. Then the sequence $\{p_j\}$ is in $\Psi_{\widehat{p}}^{\Omega}(U)$ and its limit is the origin $\mathbf{0} = (0, 0)$. For each $j$, let $q_j$ be an intersection point of the half-line $\{ p_j + (t, 0) \in \mathbb{C}^2 \mid t > 0 \}$ and $\partial(\Psi_{\widehat{p}}^{\Omega}(\Omega)) \cap \Psi_{\widehat{p}}^{\Omega}(U)$. Since $\partial\Omega$ is smooth, the point $q_j$ is uniquely determined for every sufficiently large $j$. Notice that the sequence $\{q_j\}$ also converges to the origin $\mathbf{0}$. Since the upper semicontinuity of D'Angelo type holds in this case, one can choose a subsequence $\{ q_{j_l} \}$ of $\{ q_j \}$ so that the type $\Delta(q_{j_l})$ is less than or equal to $2k$ for all $l$. For simplicity, denote this subsequence by $\{ q_j \}$ again.

Write $\Omega' := \Psi_{\widehat{p}}^{\Omega}(\Omega)$ and apply Proposition 2.3 again to the domain $\Omega'$, with ${\bf 0} \in \partial \Omega'$, at each boundary point $q_j$. Denote by $\Psi_j^{\Omega'} := \Psi_{q_j}^{\Omega'}$. Now one arrives at
\begin{equation} \nonumber
\Psi_j^{\Omega'} (\Psi_{\widehat{p}}^{\Omega}(\Omega \cap U)) = \{ \rho_j(w, z) < 0 \} \textrm{ in } \Psi_j^{\Omega'}(\Psi_{\widehat{p}}^{\Omega}(U))
\end{equation}
where:
\begin{enumerate}
\item
$\rho_j(w, z) = \textrm{Re}\,w + P_j(z, \bar{z}) + R_j(z, \bar{z}) + \textrm{Im}\left( \frac{w}{c_j} \right) Q_j \left( \textrm{Im}\left( \frac{w}{c_j} \right), z, \bar{z} \right)$,
\item
$P_j$ is a real-valued polynomial of degree $2k$ with no harmonic terms (not in general homogeneous),
\item
$R_j$ and $Q_j$ are real-valued smooth functions with the conditions $\nu(R_j) > 2k$ and $\nu(Q_j) \ge 1$,
\item
$c_j$ is a constant satisfying $\textrm{Re}\,c_j \neq 0$ and $\lim_{j \to \infty} c_j = 1$.
\end{enumerate}
Note that this local defining function is valid in $\Psi_j^{\Omega'}(\Psi_{\widehat{p}}^{\Omega}(U))$ since each $\Psi_j^{\Omega'}$ is a composition of a translation and a triangular map. Moreover, the continuity of the coordinate system $\Psi$ of Proposition 2.3 guarantees that $\Psi_j^{\Omega'}$ converges to the identity map uniformly on compact subsets of $\mathbb{C}^2$. Hence there is $J > 0$ and an open neighborhood $U'$ of the origin $\mathbf{0}$ so that
\begin{equation} \label{U'}
U' \subset \Psi_{\widehat{p}}^{\Omega}(U) \cap \Psi_j^{\Omega'}(\Psi_{\widehat{p}}^{\Omega}(U)) \textrm{ for all } j > J.
\end{equation}
Since $\rho_j$ converges to $\rho_{\widehat{p}}$ uniformly on compact subsets of $\mathbb{C}^2$, the convergences $P_j \rightarrow P$, $R_j \rightarrow R$ and $Q_j \rightarrow Q$ are also uniform on compact subsets of $\mathbb{C}^2$. We call this $\Psi_j^{\Omega'} \circ \Psi_{\widehat{p}}^{\Omega}$ as the ($j$-th) {\it centering map}.
\bigskip

\begin{figure}[!h]
\centering
\includegraphics[width=0.9\columnwidth]{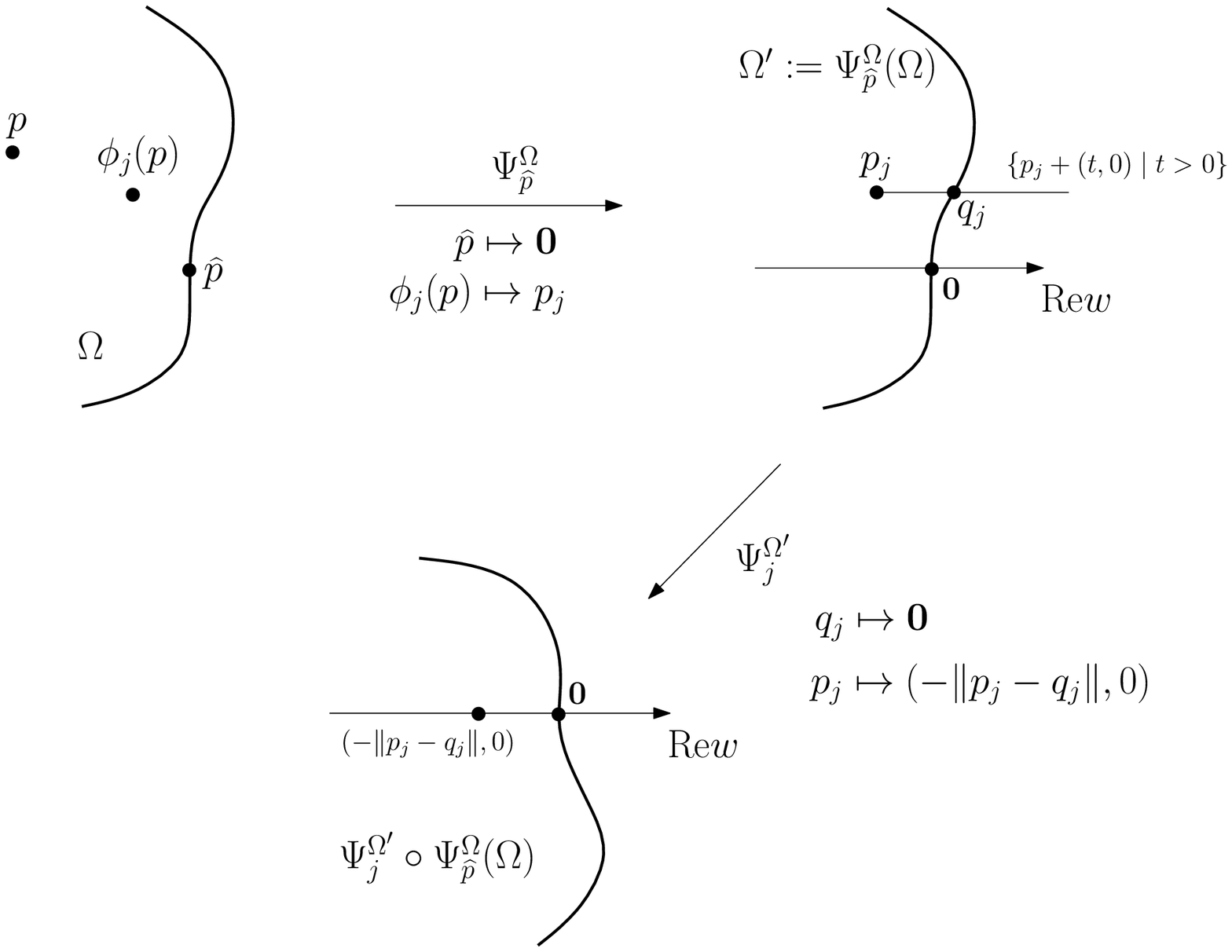}
\end{figure}
\begin{center}
{\bf Fig. 1.} The j-th centering map $\Psi_j^{\Omega'} \circ \Psi_{\widehat{p}}^{\Omega}$.
\end{center}
\bigskip
Next, we construct the sequence $\{D_j\}$ of dilations. Write $P_j (z, \bar{z}) = \sum_{n = 2}^{2k} P_{j,n} (z, \bar{z})$ where $P_{j,n}$ is a homogeneous polynomial of degree $n$. Define $\epsilon_j := \| p_j - q_j \|$ and choose $\delta_j > 0$ satisfying:
\begin{equation} \nonumber
\max \left\{ \left\| \frac{1}{\epsilon_j} P_{jn} ( \delta_j z, \overline{\delta_j z}) \right\|_{\infty}, n = 2, 3, \cdots, 2k \right\} = 1
\end{equation}
Here, the norm $\| \cdot \|_{\infty}$ is the $l^{\infty}$ norm on the space of polynomials as a finite sequence of coefficients. Let $D_j : (w, z) \mapsto \left( \frac{w}{\epsilon_j}, \frac{z}{\delta_j} \right)$ and denote by $\Lambda_j := D_j \circ \Psi_j^{\Omega'} \circ \Psi_{\widehat{p}}^{\Omega}$. Then
$$
(w, z) \in \Lambda_j (\Omega) \cap D_j (U') \Leftrightarrow (w, z) \in D_j (U') \textrm{ and } \widetilde{\rho_j}(w, z) < 0,
$$
where
$$
\widetilde{\rho_j}(w, z) = \textrm{Re}\,w + \frac{1}{\epsilon_j}P_j(\delta_j z, \delta_j \bar{z}) + \frac{1}{\epsilon_j}R_j(\delta_j z, \delta_j \bar{z}) + \textrm{Im} \left( \frac{w}{c_j} \right) Q_j \left( \textrm{Im} \left( \frac{\epsilon_j w}{c_j} \right), \delta_j z, \delta_j \bar{z} \right).
$$
\begin{lemma}
The sequence $\{ \widetilde{\rho_j} \}$ has a subsequence that converges to $\textrm{Re}\,w + \widehat{P} (z, \bar{z})$ uniformly on compact subsets of $\mathbb{C}^2$ where $\widehat{P}$ is a nonzero subharmonic polynomial of degree less than or equal to $2k$ with no harmonic terms.
\end{lemma}
\begin{proof}
The construction of $\delta_j$ and the local uniform convergence $P_j \rightarrow P$ guarantees that the sequence $\{\frac{1}{\epsilon_j}P_j(\delta_j z, \delta_j \bar{z})\}$ has a subsequence that converges to such a polynomial $\widehat{P}$. Moreover, the inequality $\epsilon_j \gtrsim \delta_j^{2k}$ holds since $P$ is homogeneous of degree $2k$. Then the term $\frac{1}{\epsilon_j} R_j(\delta_j z, \delta_j \bar{z})$ converges to the zero function owing to the conditions on $R_j$. The term $\textrm{Im} \left( \frac{w}{c_j} \right) Q_j \left( \textrm{Im} \left( \frac{\epsilon_j w}{c_j} \right), \delta_j z, \delta_j \bar{z} \right)$ converges to the zero function by a similar reason.
\end{proof}

In order to avoid using excessive indices, we shall keep $\{ \widetilde{\rho_j} \}$ for the subsequence and denote by $\widehat{\rho} := \lim \widetilde{\rho}_j$. So,
$$
\widehat{\rho} (w, z) := \textrm{Re}\,w + \widehat{P} (z, \bar{z}).
$$
Define $\widehat{\Omega} := \{ (w, z) \in \mathbb{C}^2 \mid \widehat{\rho} (w, z) < 0 \}$. Now we introduce the normal set-convergence to control the convergence of a sequence of holomorphic functions with domains varying. The following are a modification from Section 9.2.2 of \cite{gkk}.
\begin{defn} \rm
Let $\Omega_j$ be domains in $\mathbb{C}^n$ for each $j = 1, 2, \cdots$. The sequence $\Omega_j$ is said to {\it converge normally} to a domain $\widehat{\Omega}$ if the following two conditions hold.
\begin{enumerate}
\item For any compact set $K$ contained in the interior of $\bigcap_{j>m} \Omega_j$ for some positive integer $m$, $K \subset \widehat{\Omega}$.
\item For any compact subset $K'$ of $\widehat{\Omega}$, there exists a constant $m > 0$ such that $K' \subset \bigcap_{j>m} \Omega_j$.
\end{enumerate}
\end{defn}

\begin{prop}
If $\Omega_j$ is a sequence of domains in $\mathbb{C}^n$ that converges normally to the domain $\widehat{\Omega}$, then
\begin{enumerate}
\item If a sequence of holomorphic mappings $f_j : \Omega_j \rightarrow \Omega '$ from $\Omega_j$ to another domain $\Omega '$ converges uniformly on compact subsets of $\widehat{\Omega}$, then its limit is a holomorphic mapping from $\widehat{\Omega}$ into the closure of the domain $\Omega '$.
\item If a sequence of holomorphic mappings $g_j : \Omega ' \rightarrow \Omega_j$ converges uniformly on compact subsets of $\Omega '$, if $\widehat{\Omega}$ is pseudoconvex, and if there are a point $p \in \Omega '$ and a constant $c > 0$ so that the inequality $|\det{(dg_j|_p)}| > c$ holds for each $j$, then $\lim_{j \to \infty}g_j$ is a holomorphic mapping from the domain $\Omega '$ into $\widehat{\Omega}$.
\end{enumerate}
\end{prop}

Recall that $U'$ satisfying the condition (\ref{U'}), and $\Lambda_j := D_j \circ \Psi_j^{\Omega'} \circ \Psi_{\widehat{p}}^{\Omega}$. In this sense, $D_j (U')$ converges normally to $\mathbb{C}^2$. Hence the convergence $\widetilde{\rho_j} \rightarrow \widehat{\rho}$ guarantees that a sequence of domains $\{ \Lambda_j(\Omega) \}$ converges normally to $\widehat{\Omega}$. We now have constructed the sequence $\{ \Lambda_j \circ \phi_j : \Omega \rightarrow \Lambda_j(\Omega) \}$ for the proof of Theorem 2.2. Denote by $\sigma_j := \Lambda_j \circ \phi_j$.
\bigskip

\subsection{Convergence of the Pinchuk scaling sequence $\{ \sigma_j \}$} Since the limit domain $\widehat{\Omega}$ is unbounded, the convergence of $\{ \sigma_j \}$ does not follow immediately. So Pinchuk takes first the inverse sequence $\{ \sigma_j^{-1} : \sigma_j(\Omega) \rightarrow \Omega \}$. Recall that $\sigma_j(\Omega) = \Lambda_j(\Omega)$ and that the sequence $\{\sigma_j(\Omega)\}$ converges normally to $\widehat{\Omega}$. So Proposition 2.6 and Montel's theorem guarantee that there is a subsequence converging to a holomorphic map $\widehat{\tau} : \widehat{\Omega} \rightarrow \overline{\Omega}$ uniformly on compact subsets of $\widehat{\Omega}$. We shall keep the notation $\{ \sigma_j^{-1} \}$ for this subsequence. Actually, the image of $\widehat{\tau}$ is contained in $\Omega$. Indeed, suppose that there is a point in $\widehat{\Omega}$ whose image by $\widehat{\tau}$ is in $\partial \Omega$. Then $\widehat{\tau}(\widehat{\Omega}) \subset \partial\Omega$ by the pseudoconvexity of $\Omega$. This is impossible because $\widehat{\tau}(-1, 0) = p \in \Omega$. Consequently $\widehat{\tau}(\widehat{\Omega}) \subset \Omega$.

\begin{lemma}
There is a point $z_0 \in \widehat{\Omega}$ such that $d \widehat{\tau} |_{z_0}$ is nonsingular.
\end{lemma}

The proof follows the arguments of Lemma 2 in \cite{bp2}. We shall treat this in the last part of this section.
\medskip

Assume Lemma 2.7. Then the convergence of the inverse sequence $\{ \sigma_j^{-1} \}$ guarantees the uniform convergence of the sequence $\{ \det({d\sigma_j^{-1} |_{\mathbf{z}}}) \}$ on compact subsets of $\widehat{\Omega}$ by Cauchy estimates. Actually, it converges to $\det(d\widehat{\tau}|_{\mathbf{z}})$ uniformly on compact subsets of $\widehat{\Omega}$. Notice that $\det({d\sigma_j^{-1}) |_{\mathbf{z}}}$ is nowhere vanishing for any $j$ since each $\sigma_j^{-1}$ is a biholomorphic map. Hurwitz's theorem implies that $\det(d\widehat{\tau}|_{\mathbf{z}})$ is nowhere vanishing and hence $d\widehat{\tau}|_{\mathbf{z}}$ is nonsingular for all $\mathbf{z} \in \widehat{\Omega}$. In particular, $\widehat{\tau}$ is an immersion.

Suppose that $\widehat{\tau}$ is not 1-1. Then there are distinct points $s,s' \in \widehat{\Omega}$ satisfying $\widehat{\tau} (s) = \widehat{\tau} (s')$. Choose a neighborhood $U \Subset \widehat{\tau}(\widehat{\Omega})$ of $\widehat{\tau} (s)$ so that $\widehat{\tau}^{-1} (U)$ is disconnected. Let $V_{s}$ and $V_{s'}$ be mutually disjoint, connected components of $\widehat{\tau}^{-1} (U)$ such that $s \in V_{s}$ and $s' \in V_{s'}$. Note that $U$ can be adjusted so that $V_{s}$ and $V_{s'}$ are relatively compact in $\widehat{\Omega}$. Consequently, $\{ \sigma_j^{-1} |_{V_{s}} \}$ and $\{ \sigma_j^{-1} |_{V_{s'}} \}$ converge uniformly to $\widehat{\tau} |_{V_{s}}$ and $\widehat{\tau} |_{V_{s'}}$ respectively. Note that $\widehat{\tau}(V_{s}) =  \widehat{\tau}(V_{s'}) = U$, which implies that $\sigma_j^{-1} (V_{s}) \cap \sigma_j^{-1} (V_{s'}) \neq \emptyset$ for sufficiently large $j$. This contradicts the injectivity of $\sigma_j^{-1}$. This implies

\begin{prop}
$\widehat{\tau} : \widehat{\Omega} \rightarrow \Omega$ is 1-1.
\end{prop}
\medskip

Now we are concerned with the surjectivity of $\widehat{\tau}$.

\begin{lemma}
The sequence $\{ \sigma_j \}$ converges to $\widehat{\tau}^{-1}$ uniformly on compact subsets of $\widehat{\tau}(\widehat{\Omega})$.
\end{lemma}
\begin{proof}
Fix a compact set $K$ in $\widehat{\tau}(\widehat{\Omega})$ and take a compact set $K'$ such that $K \Subset K' \Subset \widehat{\tau}(\widehat{\Omega})$. Since $\widehat{\tau}^{-1}(K')$ is also compact in $\widehat{\Omega}$, $\sigma_j(\Omega)$ contains $\widehat{\tau}^{-1}(K')$ for all sufficiently large $j$. Hence the uniform convergence of the inverse scaling sequence $\{ \sigma_j^{-1} \}$ on $\widehat{\tau}^{-1}(K')$ is well-defined. So the sequence of differentials $\{ d\sigma_j^{-1} \}$ converges to $d\widehat{\tau}$ uniformly on $\widehat{\tau}^{-1}(K')$. Moreover there is $c > 1$ and $J > 0$ such that $\frac{1}{c} < |\det{(d\sigma_j^{-1})}| < c$ on $\widehat{\tau}^{-1}(K')$ for all $j > J$, since the limit map $\widehat{\tau}$ is 1-1.

Crammer's rule in Linear algebra says that
$$
(d\sigma_j^{-1})^{-1} = \frac{1}{\det({d\sigma_j^{-1}})}C_j^T
$$
where $C_j^T$ is the transpose of the cofactor matrix of $d\sigma_j^{-1}$.

Since $|\det{(d\sigma_j^{-1})}|$ has a uniformly positive lower bound on $\widehat{\tau}^{-1}(K')$, each entry of the sequence of matrices $\{ (d\sigma_j^{-1})^{-1} \}$ is uniformly bounded on $\widehat{\tau}^{-1}(K')$. Notice that $\sigma_j^{-1} \circ \widehat{\tau}^{-1}$ converges to identity uniformly on $K'$. Hence one can choose $J' > 0$ so that $K \Subset \sigma_j^{-1}(\widehat{\tau}^{-1}(K'))$ if $j > J'$. Consequently the sequence $\{ d\sigma_j \mid j > J' \}$ is uniformly bounded on $K$ by the inverse function theorem. Recall that $\sigma_j (p) = (-1, 0)$ for all $j$. Hence Montel's theorem implies that the sequence $\{ \sigma_j \}$ has a subsequence that converges to some holomorphic function $g$ uniformly on compact subsets of the interior of $K$. we shall keep the notation $\{ \sigma_j \}$ also for this convergent subsequence. Then the sequence $\{ \sigma_j^{-1} \circ \sigma_j \}$ of the identity maps converges to $\widehat{\tau} \circ g$ on $K$. Hence $g = \widehat{\tau}^{-1}$ on $K$. Since $K$ is arbitrarily chosen, $g = \widehat{\tau}^{-1}$ on all of $\widehat{\tau}(\widehat{\Omega})$ and this proves the lemma.
\end{proof}
\medskip

\begin{prop}
$\widehat{\tau} : \widehat{\Omega} \rightarrow \Omega$ is onto.
\end{prop}
\begin{proof}
Suppose that $\widehat{\tau}(\widehat{\Omega}) \subsetneq \Omega$. Choose a point $p' \in \partial (\widehat{\tau}(\widehat{\Omega})) \cap \Omega$. Since $\widehat{p} = \lim_{j \to \infty} \phi_j(p)$ is a peak point of $\Omega$, the sequence $\{ \phi_j(p') \}$ also converges to $\widehat{p}$. Now we construct the Pinchuk scaling sequence $\{ \sigma'_j := \Lambda'_j \circ \phi_j \}$ with respect to $\{ \phi_j(p') \}$ as in Section 2.2. We already observed, in Proposition 2.8, that the inverse scaling sequence $\{ {\sigma'}_j^{-1} \}$ has a subsequence converging to a 1-1 holomorphic map $\widehat{\tau'} : \widehat{\Omega}' \rightarrow \Omega$ uniformly on compact subsets of $\widehat{\Omega}'$, where $\widehat{\Omega}'$ is the limit domain of the sequence $\{ \sigma'_j (\Omega) \}$. Taking a subsequence if necessary, we may assume that the uniform convergence holds for $\sigma_j^{-1} \rightarrow \widehat{\tau}$ and ${\sigma'}_j^{-1} \rightarrow \widehat{\tau'}$ on compact subsets of $\widehat{\Omega}$ and $\widehat{\Omega}'$ respectively.

Denote by $W := \widehat{\tau}(\widehat{\Omega}) \cap \widehat{\tau'}(\widehat{\Omega}')$ and set $B_j := \sigma_j \circ {\sigma'}_j^{-1} : {\sigma'}_j(\Omega) \rightarrow \sigma_j(\Omega)$. Notice that the maps $B_j$ and $B_j^{-1}$ are polynomial automorphisms (in fact, essentially triangular) of $\mathbb{C}^2$ with degree less than or equal to $2k$. Fix a nonempty open set $W' \Subset W$. By Lemma 2.9, $\sigma_j|_{W'}$ converges uniformly to $\widehat{\tau}^{-1}|_{W'}$ and hence $\lim_{j \to \infty} B_j|_{\widehat{\tau'}^{-1}(W')} \equiv \widehat{\tau}^{-1} \circ \widehat{\tau'}|_{\widehat{\tau'}^{-1}(W')}$, a biholomorphism of ${\widehat{\tau'}^{-1}(W')}$ and $\widehat{\tau}(W')$. Since each $B_j$ is a polynomial automorphism of $\mathbb{C}^2$ and of degree less than or equal to $2k$, it converges uniformly on compact subsets of $\mathbb{C}^2$ to a polynomial map $\widehat{B}$ of degree less than or equal to $2k$. Actually the limit map $\widehat{B}$ is in $Aut(\mathbb{C}^2)$ by a similar argument as in the proof of an injectivity of $\widehat{\tau}$. Similarly, $B_j^{-1}$ converges to $\widehat{B}^{-1}$. Now Proposition 2.6 guarantees that $\widehat{B}|_{\widehat{\Omega}'} : \widehat{\Omega}' \rightarrow \widehat{\Omega}$ and $\widehat{B}^{-1}|_{\widehat{\Omega}} : \widehat{\Omega} \rightarrow \widehat{\Omega}'$ are inverse to each other and hence $\widehat{\Omega}$ and $\widehat{\Omega}'$ are biholomorphic. Notice that $\widehat{B}|_{\widehat{\Omega}'}(-1, 0) \in \widehat{\Omega}$ and $\widehat{\tau} \circ \widehat{B}|_{\widehat{\Omega}'}(-1, 0) = p'$. Therefore $p' \in \widehat{\tau}(\widehat{\Omega})$. This contradicts that $p' \in \partial (\widehat{\tau}(\widehat{\Omega})) \cap \Omega$.
\end{proof}
\medskip

\begin{rmk} \rm
The Pinchuk scaling sequence also depends on the initial point. But the argument above shows that their limit domains have to be equivalent, via a polynomial automorphism of $\mathbb{C}^2$ with its degree not more than $2k$, the type of the orbit accumulating boundary point.
\end{rmk}
\medskip

Propositions 2.8 and 2.10 imply that the limit map $\widehat{\tau} : \widehat{\Omega} \rightarrow \Omega$ of the inverse scaling sequence is a biholomorphic map. Consequently, the limit map $\widehat{\sigma}$ of the Pinchuk scaling sequence is defined on all of $\Omega$ and it satisfies $\widehat{\sigma} \equiv \widehat{\tau}^{-1}$. So the only remaining part for the proof of Theorem 2.2 is justifying Lemma 2.7.
\bigskip

{\bf Proof of Lemma 2.7.} The main idea is an estimate of the invariant metric introduced by Sibony \cite{s}. It is defined by
\begin{equation} \nonumber
F_S(p, \xi ; M) := \sup_u \left\{ \left( \sum_{i, j = 1}^{n} \frac{\partial^2 u}{\partial z_i \partial \bar{z}_j} \bigg|_p \xi_i \bar{\xi}_j \right)^{\frac{1}{2}} \colon u \in A_{M} (p) \right\},
\end{equation}
where $A_{M} (p)$ is the set of all plurisubharmonic functions of $M$ defined as follows: $u \in A_{M} (p)$ if $0 \le u \le 1$, $u(p) = 0$, $u \in C^2$ near $p$, and $\log u$ is plurisubharmonic on $M$.
\medskip

For this metric, Sibony proved;
\begin{thm}[\cite{s}]
Let $M$ be a complex manifold. If there is a bounded plurisubharmonic function $u$ of $M$ and there is a constant $\delta > 0$ such that $dd^c u$ is positive definite on a $\delta$-neighborhood of $p$, then there is an $\epsilon = \epsilon(\delta) > 0$ such that $F_S(p, \xi ; M) \ge \epsilon |\xi|$ for all $\xi$ in the holomorphic tangent space $T_p^{\mathbb{C}} M$.
\end{thm}

Recall that $\Psi_{\widehat{p}}^{\Omega}(\Omega)$ is a bounded domain in $\mathbb{C}^2$ and its local defining function at $\mathbf{0}$ is $\rho_{\widehat{p}}$. Hence the main theorem of \cite{df} by Diederich and Forn{\ae}ss says that there is a $C^{\infty}$ defining function $\beta$ and a sufficiently small $\eta > 0$ such that $-(-\beta)^{\eta}$ is a strictly plurisubharmonic bounded exhaustion function of $\Psi_{\widehat{p}}^{\Omega}(\Omega)$. Since both $\beta$ and $\rho_{\widehat{p}}$ are defining functions of $\Psi_{\widehat{p}}^{\Omega}(\Omega)$, $\lim_{{\bf z} \rightarrow {\bf 0}} \left(\frac{\beta({\bf z})}{\rho_{\widehat{p}} ({\bf z})}\right)^{\eta} = c$ for some positive constant $c$. Taking a constant multiple of $\beta$, one may assume that $c = 1$. Define $\widetilde{\beta}_j := -\epsilon_j^{-\eta}(-\beta \circ (\Psi_j^{\Omega'})^{-1} \circ D_j^{-1})^{\eta}$. Since the equality $-\widetilde{\rho_j} = \epsilon_j^{-1} (-\rho_{\widehat{p}} \circ (\Psi_j^{\Omega'})^{-1} \circ D_j^{-1})$ holds due to the construction of $\widetilde{\rho_j}$ and $\rho_{\widehat{p}}$, we have the following formula:
\begin{equation} \nonumber
\widetilde{\beta}_j = -\left(\frac{\beta}{\rho_{\widehat{p}}} \circ (\Psi_j^{\Omega'})^{-1} \circ D_j^{-1} \right)^{\eta} \cdot (-\widetilde{\rho_j})^{\eta}.
\end{equation}
Recall that $\widetilde{\rho_j} \rightarrow \widehat{\rho}$ uniformly on compact sets of $\mathbb{C}^2$. Since $(\Psi_j^{\Omega'})^{-1} \circ D_j^{-1}$ converges to ${\bf 0}$ uniformly on compact sets of $\mathbb{C}^2$, the convergence $\widetilde{\beta}_j \rightarrow -(-\widehat{\rho})^{\eta}$ is also uniform on compact sets of $\mathbb{C}^2$. Note that the plurisubharmonicity of each $\widetilde{\beta}_j$ guarantees the plurisubharmonicity of the limit map $-(-\widehat{\rho})^{\eta}$. Recall that $\widehat{\rho} (w, z) = \textrm{Re}\,w + \widehat{P} (z, \bar{z})$ for a certain nonzero subharmonic polynomial $\widehat{P}$ without harmonic terms. Hence $e^{-(-\widehat{\rho})^{\eta}}$ is strictly plurisubharmonic almost everywhere on $\widehat{\Omega}$. In particular, there is a point $q_0 \in \widehat{\Omega}$ such that $e^{-(-\widehat{\rho})^{\eta}}$ is strictly plurisubharmonic at $q_0$. So $e^{\widetilde{\beta}_j}$ is uniformly strictly plurisubharmonic at $q_0$ for all sufficiently large $j$. Now Theorem 2.11 guarantees that there is an $\epsilon > 0$ such that $F_S(q_0, \xi; \Lambda_j(\Omega)) \ge \epsilon |\xi|$ holds whenever $\xi \in \mathbb{C}^2$ and $j$ is sufficiently large. Note that this metric is invariant under biholomorphic transformations, and hence
\begin{equation} \nonumber
F_S(\tau_j(q_0), d \tau_j |_{q_0} (\xi); \Omega) \ge \epsilon |\xi|
\end{equation}
for all $\xi \in \mathbb{C}^2$ and $j$ sufficiently large.
\medskip

Let $q' \in \Omega$ be the limit point of $\{ \tau_j(q_0) \}$. Then the uniform convergence $d \tau_j |_{q_0} \rightarrow d \widehat{\tau} |_{q_0}$ on compact subsets of $\widehat{\Omega}$ implies that $F_S(q', d \widehat{\tau} |_{q_0} (\xi); \Omega) \ge \epsilon |\xi|$ for all $\xi \in \mathbb{C}^2$. Consequently, $d \widehat{\tau}$ is nonsingular at $q_0$ and this proves Lemma 2.7.
\medskip

\section{The Frankel scaling sequence}

\begin{defn}[\cite{f}] \rm
Let $\Omega$ be a domain in $\mathbb{C}^n$ with a point $p \in \Omega$, and $\{ \phi_j \}$ be a sequence in $Aut(\Omega)$. Then the {\it Frankel scaling sequence} $\{ \omega_j : \Omega \rightarrow \mathbb{C}^n \}$ with respect to $(\Omega, p, \{\phi_j\})$ is defined by
\begin{equation}
\omega_j (\mathbf{z}) := [d\phi_j |_p]^{-1} (\phi_j(\mathbf{z}) - \phi_j (p))
\end{equation} \nonumber
where $\mathbf{z} = (z_1, \cdots, z_n) \in \mathbb{C}^n$.
\end{defn}

Notice that $\omega_j(p) = 0$ and $d\omega_j|_p = I$ for all $j$, where $I$ is the identity map. So its construction appears to be more intrinsic than Pinchuk's. However, the convergence is known only for the convex Kobayashi hyperbolic domains \cite{f}. In fact, we show here that there is a non-convex domain for which the Frankel scaling sequence diverges, even though the domain is biholomorphic to a bounded convex domain.

\begin{ex} \rm
Let $\Omega_1 = \{ (w, z) \in \mathbb{C}^2 \mid \textrm{Re}\,w + |z|^4 < 0 \}$. Note that $\Omega_1$ is biholomorphic to the Thullen domain $\{ |w|^2 + |z|^4 < 1 \}$, which is bounded and convex. Fix a point $p = (-1, 0) \in \Omega_1$. Consider a sequence $\left\{ \phi_j (w, z) := \left( \frac{1}{j^4}w, \frac{1}{j}z \right) \right\}$ of automorphisms of $\Omega_1$. Then the Frankel scaling map $\omega_j^{\Omega_1}$ with respect to $(p, \phi_j)$ is
\begin{equation} \nonumber
\omega_j^{\Omega_1} (w, z) = [d\phi_j |_p]^{-1} (\phi_j(w, z) - \phi_j (p)) = (w+1, z).
\end{equation}
Consequently, the sequence $\{ \omega_j^{\Omega_1} (w, z) \}$ converges to $(w+1, z)$.
\medskip

On the other hand, let $\psi$ be defined by $\psi(w, z) = (w - 2z^2, z)$ and $\Omega_2 := \psi(\Omega_1)$. Notice that $\Omega_2 = \{ (w, z) \in \mathbb{C}^2 \mid \textrm{Re}\,w + z^2 + \overline{z}^2 + |z|^4 < 0 \}$ and $\psi(p) = p$. Denote by $\widetilde{\phi_j} := \psi \circ \phi_j \circ \psi^{-1}$ which is in $Aut(\Omega_2)$. Then the Frankel scaling map $\omega_j^{\Omega_2}$ with respect to $(p, \widetilde{\phi_j})$ is
\begin{equation} \nonumber
\omega_j^{\Omega_2} (w, z) = [d\widetilde{\phi_j} |_{\psi(p)}]^{-1} (\widetilde{\phi_j}(w, z) - \widetilde{\phi_j}(p))
= \left(\begin{array}{cc} w + 2(1 - j^2)z^2 + 1 \\ z \end{array} \right).
\end{equation}

Observe that every subsequence of $\{ \omega_j^{\Omega_2} (w, z) \}$ diverges.
\end{ex}

\begin{rmk} \rm
On the other hand, Frankel's scaling sequence can converge on a certain nonconvex domain. In such a case, the convergence is preserved through $\mathbb{C}$-affine biholomorphic transformations.

Let $\Omega$, $p \in \Omega$ and $\{\phi_j\} \subset Aut(\Omega)$ be given as in Definition 3.1. Let $\psi$ be a nonsingular $\mathbb{C}$-affine map and denote by $\widetilde{\phi_j} := \psi \circ \phi_j \circ \psi^{-1}$. Then the Frankel scaling map with respect to $(\psi(\Omega), \psi(p), \widetilde{\phi_j})$ is
\begin{equation} \nonumber
 [d \widetilde{\phi_j} |_{\psi (p)}]^{-1} (\widetilde{\phi_j}(z) - \widetilde{\phi_j} (\psi (p))) = d \psi_p [d \phi_j |_p]^{-1} (\phi_j (\psi^{-1}(z)) - \phi_j (p)).
\end{equation}
Notice that the right hand side is the composition of a nonsingular matrix $d \psi_p$ and the Frankel scaling map with respect to $(\Omega, p, \phi_j)$. So the convergence is invariant under the nonsingular affine transformations.
\end{rmk}

In Example 3.2, observe that $\psi^{-1}$ removes the harmonic term $z^2 + \bar{z}^2$, whose vanishing order is smaller than the principle term $|z|^4$, in the expression of the defining function of $\Omega_2$. So one would hope that taking a coordinate change map $\Psi \in Aut(\mathbb{C}^2)$, that removes harmonic terms from the defining function of given domain at the limit point of the sequence $\{\phi_j(p)\}$, may be enough to make the (adjusted) Frankel scaling sequence converge. However, the following example shows that this is not true in general.

\begin{ex} \rm
Let $\Omega_3 = \{ (w, z) \in \mathbb{C}^2 \mid \textrm{Re}\,w + 4z\overline{z}^3 + 6|z|^4 + 4z^3 \overline{z} < 0 \}$ and fix the point $p = (1, i) \in \Omega_3$. We consider the map
\begin{equation} \nonumber
\begin{split}
&\phi_{\mu}(w, z)
\\
=& \left(\frac{1}{\mu^8}w + \frac{8i(\mu-1)}{\mu^8}z^3 - \frac{12(\mu-1)^2}{\mu^8}z^2 - \frac{8i(\mu-1)^3}{\mu^8}z + \frac{2(\mu-1)^4}{\mu^8}, \frac{1}{\mu^2}z + \frac{\mu i - i}{\mu^2}\right).
\end{split}
\end{equation}
Then $\{\phi_{\mu}\}_{\mu} \subset Aut(\Omega_3)$. Notice that $\phi_{\mu}(p)$ converges to the origin $\mathbf{0}$ as $\mu$ goes to infinity. So we investigate whether no adjustments are needed since the defining function of $\Omega_3$ has no harmonic terms in $z$. Observe that the Frankel scaling map $\omega_{\mu}^{\Omega_3}$ with respect to $(p, \phi_{\mu})$ is:
$$
\omega_{\mu}^{\Omega_3}(w, z) = (w + 8i(\mu-1)z^3 - 12(\mu-1)^2z^2 + 24i\mu(\mu-1)z + 12\mu^2 - 8\mu - 5, z - i).
$$
It is now clear that every subsequence of $\{ \omega_{\mu}^{\Omega_3} \}$ diverges.
\end{ex}
In the light of these examples, the following question arises naturally: {\it for a given $(\Omega, p, \phi_j)$ as in Definition 3.1, would there exist a map $\Psi \in Aut(\mathbb{C}^2)$ such that the Frankel scaling sequence with respect to $(\Psi(\Omega), \Psi(p), \Psi \circ \phi_j \circ \Psi^{-1})$ converges?} In Example 3.2, the map $\psi^{-1}$ performs the role for $(\Omega_2, p, \widetilde{\phi_j})$, while finding such a map is not easy in Example 3.3. We leave this for a later study. However, the following adjustment can be a reasonable alternative.
\medskip

\begin{defn} \rm
Let $\Omega$, $p$ and $\{\phi_j\}$ be as in Definition 3.1. Consider the sequence $\{ \psi_j \} \subset Aut(\mathbb{C}^n)$ converging to another automorphism $\widehat{\psi}$ uniformly on compact sets of $\mathbb{C}^n$. Now we define the {\it modification} (i.e., {\it modified Frankel scaling sequence})
\begin{equation} \nonumber
\omega_j (\mathbf{z}) := [d(\psi_j \circ \phi_j \circ \psi_j^{-1}) |_{\psi_j(p)}]^{-1} (\psi_j \circ \phi_j \circ \psi_j^{-1}(\mathbf{z}) - \psi_j \circ \phi_j (p))
\end{equation}
by $\{\psi_j\}$ of the original Frankel scaling sequence.
\end{defn}

\begin{thm}
Let $\Omega$ be a bounded domain in $\mathbb{C}^n$ which admits a boundary accumulation automorphism orbit $\{\phi_j(p)\}$. Assume that there is a sequence $\{ \psi_j \} \subset Aut(\mathbb{C}^n)$ satisfying:
\begin{enumerate}
\item $\{\psi_j\}$ converges to $\widehat{\psi} \in Aut(\mathbb{C}^n)$ uniformly on compact sets of $\mathbb{C}^n$.
\item There is a sequence $\{ D_j\}$ of $\mathbb{C}$-affine maps so that the sequence $\{ D_j \circ \psi_j \circ \phi_j \}$ converges to a certain biholomorphism-into $\mathbb{C}^n$ uniformly on compact subsets of $\Omega$.
\end{enumerate}
Then the modified Frankel scaling sequence by $\{\psi_j\}$ has a subsequence that converges to a biholomorphism-into $\mathbb{C}^n$ uniformly on compact subsets of $\widehat{\psi}(\Omega)$.
\end{thm}

\begin{proof}
Let $\sigma_j := D_j \circ \psi_j \circ \phi_j$ and $\widehat{\sigma} := \lim_{j \to \infty} \sigma_j$. Let $\omega_j$ represent the modified Frankel scaling map by $\psi_j$. Write $\widetilde{\phi_j} := \psi_j \circ \phi_j \circ \psi_j^{-1}$, then
\begin{equation} \nonumber
\omega_j (\mathbf{z}) := [d\widetilde{\phi_j} |_{\psi_j(p)}]^{-1} (\widetilde{\phi_j}(\mathbf{z}) - \psi_j \circ \phi_j(p)).
\end{equation}
Define the $\mathbb{C}$-affine map $A_j$ of $\mathbb{C}^n$ by
\begin{equation} \nonumber
A_j (\mathbf{z}) := [d\widetilde{\phi_j} |_{\psi_j(p)}]^{-1} (D_j^{-1} (\mathbf{z}) - \psi_j \circ \phi_j(p)).
\end{equation}
Now, $A_j$ makes the following diagram commute.
\bigskip

\begin{figure}[!h]
\centering
\includegraphics[width=0.6\columnwidth]{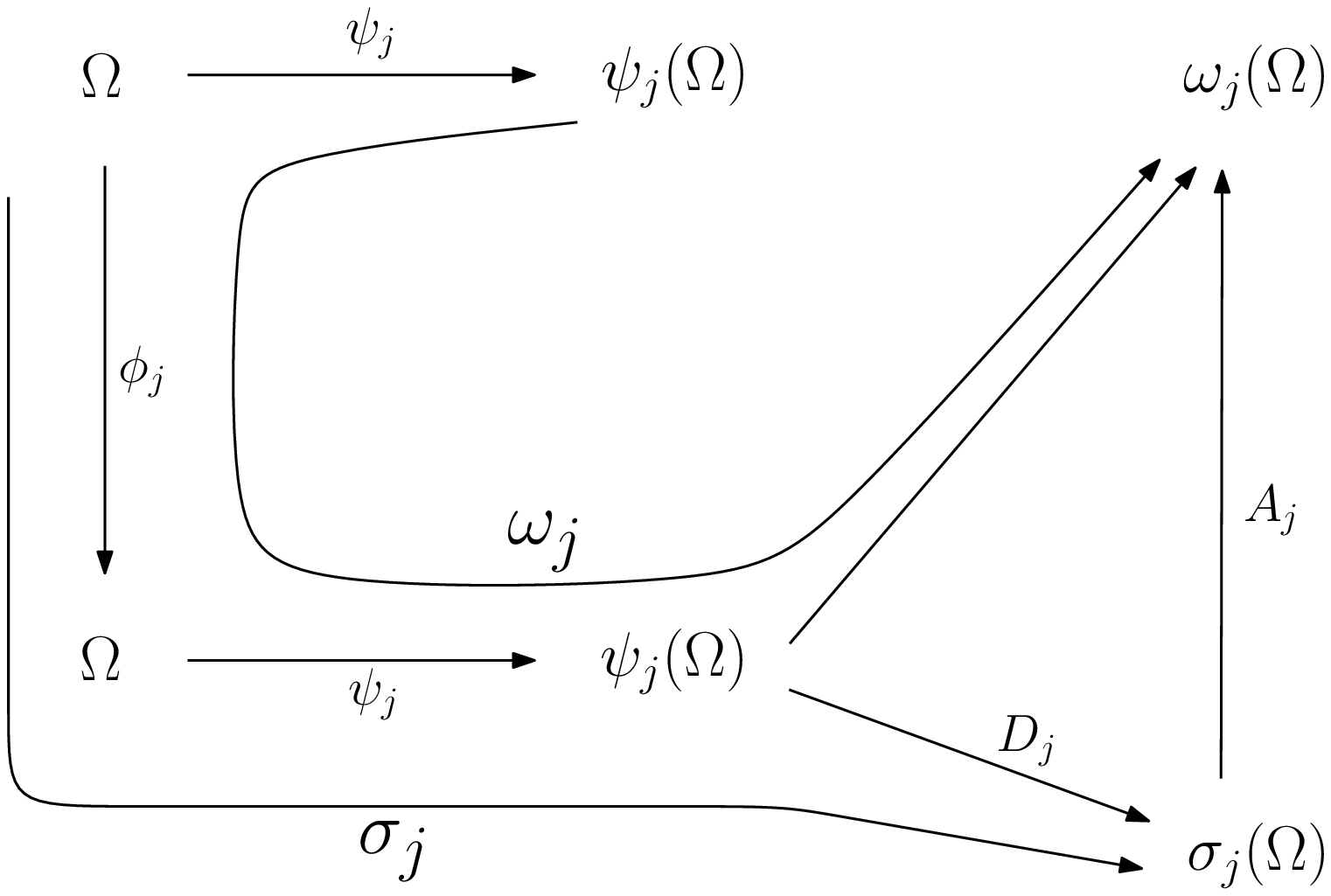}
\end{figure}
\begin{center}
{\bf Fig. 2.} A relationship of j-th scaling maps.
\end{center}
\bigskip
Notice that the map $A_j$ enjoys the following properties:
\begin{itemize}
\item $A_j |_{\sigma_j(\Omega)} = \omega_j \circ \psi_j \circ \sigma_j^{-1}$.
\item $A_j(\sigma_j(p)) = \mathbf{0}$, for all $j$.
\item $dA_j$ converges to a nonsingular matrix.
\end{itemize}
The first two properties follow directly from the construction of $A_j$, and the third is nothing but
\begin{equation} \nonumber
\lim_{j \rightarrow \infty} dA_j|_{\sigma_j(p)} = \lim_{j \rightarrow \infty} d\omega_j |_{\psi_j(p)} \circ d\psi_j |_p \circ d\sigma_j^{-1} |_{\sigma_j(p)} = d\widehat{\psi}|_p \circ d\widehat{\sigma}^{-1} |_{\widehat{\sigma}(p)}.
\end{equation}
Hence the sequence $\{A_j\}$ converges to some nonsingular $\mathbb{C}$-affine map, say $\widehat{A}$, satisfying $\widehat{A}(\widehat{\sigma}(p)) = \mathbf{0}$. Notice that $\omega_j = A_j \circ \sigma_j \circ \psi_j^{-1}$ and $\psi_j(\Omega)$ converges to $\widehat{\psi}(\Omega)$ in the sense of normal set convergence. Hence the uniform convergence of sequences $\{A_j\}$, $\{\sigma_j\}$ and $\{\psi_j^{-1}\}$ implies that the sequence $\{\omega_j\}$ converges to $\widehat{A} \circ \widehat{\sigma} \circ \widehat{\psi}^{-1}$ uniformly on compact subsets of $\widehat{\psi}(\Omega)$. This proves the theorem.
\end{proof}

\begin{rmk} \rm
If $\Omega$ is convex, then $\psi_j$ satisfying the hypotheses (1) and (2) automatically exist, affine maps, as demonstrated in \cite{kk}. It is clear now that Theorem 3.6, with Remark 3.3, generalizes the convergence theorem of Frankel for convex domains.
\end{rmk}
\medskip

Recall that the sequence of centering maps $\{\Psi_j^{\Omega'} \circ \Psi_{\widehat{p}}^{\Omega}\} \subset Aut(\mathbb{C}^2)$ constructed in Section 2.2. Combined with Theorem 2.2, Theorem 3.6 implies the following:

\begin{thm}
Let $\Omega$ be a bounded domain in $\mathbb{C}^2$ with smooth pseudoconvex boundary. Assume that $Aut(\Omega)$ admits a boundary accumulating automorphism orbit $\{\phi_j(p)\}$ converging to $\widehat{p} \in \partial\Omega$. If $\widehat{p}$ is of finite type in the sense of D'Angelo, then there exists a modified Frankel scaling sequence, by $\{\Psi_j^{\Omega'} \circ \Psi_{\widehat{p}}^{\Omega}\}$, that converges to a 1-1 holomorphic map into $\mathbb{C}^n$ uniformly on compact subsets of $\Psi_{\widehat{p}}^{\Omega}(\Omega)$.
\end{thm}
\bigskip

\section{Equivalence of two scalings}

Theorem 3.6 above can be restated as follows: If the Pinchuk scaling sequence built upon the sequence of the global centering maps $\{ \psi_j \}$ converges to $\widehat{\sigma}$, then the modified Frankel scaling sequence by $\{ \psi_j \}$ also converges to $\widehat{\omega}$, say.

Observe that $\widehat{\omega} = \widehat{A} \circ \widehat{\sigma} \circ \widehat{\psi}^{-1}$. Therefore we have

\begin{thm}
If the Pinchuk scaling sequence, built upon the sequence of the global centering maps $\{ \psi_j \}$, converges to $\widehat{\sigma}$. Then the modified Frankel scaling sequence by $\{ \psi_j \}$ converges to $\widehat{\omega}$, and
\begin{equation} \nonumber
\widehat{\omega} = \widehat{A} \circ \widehat{\sigma} \circ \widehat{\psi}^{-1}
\end{equation}
for some nonsingular $\mathbb{C}$-affine map $\widehat{A}$ and the limit map $\widehat{\psi}$ of $\{ \psi_j \}$.
\end{thm}

\begin{rmk} \rm
If the given domain is bounded and convex, Kim and Krantz \cite{kk} proved that Pinchuk's scaling sequence converges, that $\widehat{\psi}$ turns out to be a $\mathbb{C}$-affine biholomorphism, and consequently that the limits of two scalings are $\mathbb{C}$-affinely equivalent. In this regard, Theorem 4.1 generalizes their result to a nonconvex case.
\end{rmk}
\bigskip

\section{Appendix}

We now prove Proposition 2.3. For convenience, we restate the proposition.

\keylemma*

\begin{proof}
Define the translation map by $T_q(z) := z-q$. Since the outward unit normal vector of $\partial\Omega$ at $p$ is $(1, 0)$, the implicit function theorem guarantees that there is a neighborhood $U$ of $p$ such that
$$
T_p(\Omega \cap U) = \{ (w, z) \mid \textrm{Re}\,w + P_p (z, \bar{z}) + \textrm{Im}\,w \cdot Q_p (\textrm{Im}\,w, z, \bar{z}) < 0 \}
$$
for some real valued smooth function $P_p$ and $Q_p$ with the conditions $\nu(P_p) \ge 2$ and $\nu(Q_p) \ge 1$. Now, one can obtain
$$
T_q(\Omega \cap U) = \{ (w, z) \mid \textrm{Re}\,w + P_q (z, \bar{z}) + b_q \cdot \textrm{Im}\,w + \textrm{Im}\,w \cdot Q_q (\textrm{Im}\,w, z, \bar{z}) < 0 \}
$$
for some real constant $b_q$ and real valued smooth functions $P_q$ and $Q_q$ with the conditions $\nu(P_q) \ge 1$ and $\nu(Q_p) \ge 1$. Notice that all of $P_q$, $Q_q$ and $b_q$ vary continuously with respect to $q \in \partial \Omega \cap U$, and hence $\lim_{q \to p} P_q = P_p$, $\lim_{q \to p} Q_q = Q_p$ and $\lim_{q \to p} b_q = 0$. Now we take the coordinate change $S_q : (w, z) \mapsto ((1-ib_q)w, z)$ on $T_q(\Omega)$. Denote by $c_q := (1-ib_q)$. Then
$$
S_q \circ T_q(\Omega \cap U) = \left\{ (w, z) \,\Big|\, \textrm{Re}\,w + P_q (z, \bar{z}) + \textrm{Im} \left( \frac{w}{c_q} \right) Q_q \left(\textrm{Im} \left( \frac{w}{c_q} \right), z, \bar{z} \right) < 0 \right\}.
$$
Define $H_{q, 1} (z, \bar{z})$ by the harmonic part of $P_q (z, \bar{z})$ of degree $r$ as follows:
$$
H_{q, 1} (z, \bar{z}) := 2 \textrm{Re}\,h_{q, 1} (z) \textrm{ where } h_{q, 1}(z) := \sum_{j = 1}^{r} \frac{\partial^j P_q (z, \bar{z})}{\partial z^j} \bigg|_0 \cdot \frac{z^j}{j!}.
$$
Let $P_{q, 1} := P_q - H_{q, 1}$, then $P_{q, 1} (z, \bar{z})$ has no harmonic terms of degree less than or equal to $r$. Now consider the coordinate change $\psi_{q, 1} : (w, z) \mapsto (w + 2h_{q, 1} (z), z)$. Then,
\begin{equation} \nonumber
\begin{split}
\psi_{q, 1} &(S_q(T_q(\Omega \cap U)))
\\
=& \bigg\{ \textrm{Re}\,w + P_{q, 1} (z, \bar{z}) + \textrm{Im} \left( \frac{-2h_{q, 1} (z)}{c_q} \right) Q_q \left( \textrm{Im} \left( \frac{w - 2h_{q, 1} (z)}{c_q} \right), z, \bar{z} \right)
\\
&\quad + \textrm{Im} \left( \frac{w}{c_q} \right) Q_q \left( \textrm{Im}\left( \frac{w - 2h_{q, 1} (z)}{c_q} \right), z, \bar{z} \right) < 0 \bigg\}.
\end{split}
\end{equation}
Denote by
$$
R_{q, 1} \left( \textrm{Im} \left( \frac{w}{c_q} \right), z, \bar{z} \right) := \textrm{Im} \left( \frac{ - 2h_{q, 1} (z)}{c_q} \right) Q_q \left( \textrm{Im} \left( \frac{ w - 2h_{q, 1} (z)}{c_q} \right), z, \bar{z} \right),
$$
and
\begin{equation} \nonumber
\begin{split}
Q_{q, 1} \left( \textrm{Im} \left( \frac{w}{c_q} \right), z, \bar{z} \right) :=& \, Q_q \left( \textrm{Im}\left( \frac{w - 2h_{q, 1} (z)}{c_q} \right), z, \bar{z} \right)
\\
+& \left( \textrm{Im} \left( \frac{w}{c_q} \right) \right)^{-1} \left( R_{q, 1} \left( \textrm{Im} \left( \frac{w}{c_q} \right), z, \bar{z} \right) - R_{q, 1} (0, z, \bar{z}) \right).
\end{split}
\end{equation}
Then $R_{q, 1}$ and $Q_{q, 1}$ vary continuously with respect to $q$ and the local defining function of $\psi_{q, 1} (S_q(T_q(\Omega \cap U)))$ above can be rewritten as follows:
\begin{equation} \nonumber
\begin{split}
&\psi_{q, 1} \circ S_q \circ T_q(\Omega \cap U)
\\
=& \bigg\{ (w, z) \,\Big|\, \textrm{Re}\,w + P_{q, 1} (z, \bar{z}) + R_{q, 1} (0, z, \bar{z}) + \textrm{Im} \left( \frac{w}{c_q} \right) Q_{q, 1} \left( \textrm{Im} \left( \frac{w}{c_q} \right), z, \bar{z} \right) < 0 \bigg\}.
\end{split}
\end{equation}
Notice that $\nu(R_{q, 1} (0, z, \bar{z})) \ge \nu(h_{q, 1}(z)) + \nu(Q_q) > 1$. Now we construct the following functions, varying continuously with respect to $q$, inductively for $j = 1, 2, \cdots, r$.

\begin{itemize}
\item $R_{q, 0} := P_q$, $Q_{q, 0} := Q_q$.
\item $H_{q, j} (z, \bar{z}) := 2 \textrm{Re}\,h_{q, j} (z) \textrm{ where } h_{q, j}(z) := \sum_{k = j}^{r} \frac{\partial^k R_{q, j-1} (z, \bar{z})}{\partial z^k} \big|_0 \cdot \frac{z^k}{k!}$.
\item $P_{q, j} := R_{q, j-1} - H_{q, j}$.
\item $\psi_{q, j} : (w, z) \mapsto (w + 2h_{q, j} (z), z)$
\item $R_{q, j} \left( \textrm{Im} \left( \frac{w}{c_q} \right), z, \bar{z} \right) := \textrm{Im} \left( \frac{- 2h_{q, j} (z)}{c_q} \right) Q_{q, j-1} \left( \textrm{Im} \left( \frac{ w-2h_{q, j} (z)}{c_q} \right), z, \bar{z} \right)$.
\item $Q_{q, j} \left( \textrm{Im} \left( \frac{w}{c_q} \right), z, \bar{z} \right) := Q_{q, j-1} \left( \textrm{Im}\left( \frac{w - 2h_{q, j} (z)}{c_q} \right), z, \bar{z} \right)$ \newline \phantom{$Q_{q, j} \left( \textrm{Im} \left( \frac{w}{c_q} \right), z, \bar{z} \right)$} $+ \left( \textrm{Im} \left( \frac{w}{c_q} \right) \right)^{-1} \left( R_{q, j} \left( \textrm{Im} \left( \frac{w}{c_q} \right), z, \bar{z} \right) - R_{q, j} (0, z, \bar{z}) \right)$.
\end{itemize}
Define $\psi_q := \psi_{1, r} \circ \cdots \circ \psi_{q, 1}$. Then $\psi_q$ varies continuously with respect to $q$ and satisfies:
\begin{equation} \nonumber
\begin{split}
\psi_q \circ S_q \circ T_q(\Omega \cap U) = \bigg\{ (w, z) \,\Big|\, &\textrm{Re}\,w + \Sigma_{j=1}^{r} P_{q, j} (z, \bar{z}) + R_{q, r}(0, z, \bar{z}) 
\\
&+ \textrm{Im} \left( \frac{w}{c_q} \right) Q_{q, r} \left( \textrm{Im} \left( \frac{w}{c_q} \right), z, \bar{z} \right) < 0 \bigg\}.
\end{split}
\end{equation}
Notice that $\Sigma_{j=1}^{r} P_{q, j} (z, \bar{z}) + R_{q, r}(0, z, \bar{z})$ has no harmonic terms of degree less than or equal to $r$, $Q_{q, r}$ has no constant term, and $\nu(R_{q, r}(0, z, \bar{z})) > r$. Set $\Psi_q := \psi_q \circ S_q \circ T_q$, $P := \Sigma_{j=1}^{r} P_{q, j}$, $R(z, \bar{z}) := R_{q, r}(0, z, \bar{z})$ and $Q := Q_{q, r}$. Then the first assertion of the proposition follows.
\smallskip

To prove the second, fix $q \in \partial\Omega \cap U$ and $r = 2k$, the type of $q$. Denote $\rho_q(w, z)$ by the local defining function of $\Psi_q (\Omega)$ near $q$ as in Proposition 2.3. Recall that $P$ has no harmonic terms. So it is sufficient to prove that $\nu(P) = 2k$.

Suppose $\nu(P) > 2k$. Then for the analytic disc $\delta_1 : z \mapsto (0, z)$, the vanishing order of a composition $\rho_q \circ \delta_1$ is larger than $2k$ and this contradicts that $\Delta(p) = 2k$. So $\nu(P) \le 2k$.

Now assume that $\nu(P) < 2k$. Then there is a type-realizing holomorphic disc $\delta_2 : z \mapsto (f(z), g(z))$, so that $\frac{\nu(\rho_q \circ \delta_2)}{\nu(\delta_2)} = 2k$. If $\nu(f) \le \nu(g)$ then $f$ is not identically zero, and hence
\begin{equation} \nonumber
\frac{\nu(\rho_q \circ \delta_2)}{\nu(\delta_2)} = \phantom{.} \frac{\min \{ \nu(\textrm{Re}\,f), \nu(P) \nu(g) \} }{\nu(f)} \le 1.
\end{equation}
The first equality holds since $P$ has no harmonic terms while $\textrm{Re}\,f$ is harmonic. But this is impossible because $\frac{\nu(\rho_q \circ \delta_2)}{\nu(\delta_2)} = 2k$. So $\nu(f) > \nu(g)$ and hence
\begin{equation} \nonumber
\frac{\nu(\rho_q \circ \delta_2)}{\nu(\delta_2)} = \frac{\min \{ \nu(\textrm{Re}\,f), \nu(P) \nu(g) \} }{\nu(g)} \le \nu(P) < 2k.
\end{equation}
The last inequality holds due to the assumption. However, this contradicts that $\frac{\nu(\rho_q \circ \delta_2)}{\nu(\delta_2)} = 2k$ and hence $\nu(P) = 2k$. Therefore the second assertion of the proposition follows.
\end{proof}

\end{document}